\documentclass[11pt, a4]{amsart}
\usepackage{amssymb, amsthm, amscd}
\usepackage{tikz}
\usepackage{tikz-cd}
\usepackage{atbegshi}
\usepackage[driverfallback=dvipdfmx]{hyperref}
\theoremstyle{definition}
\newtheorem{thm}{\bf Theorem}
\newtheorem{prop}{Proposition}
\newtheorem{lem}{Lemma}

\newtheorem{defn}{\bf{Definition}}

\newtheorem*{examples*}{Examples}

\theoremstyle{remark}

\renewcommand{\epsilon}{\varepsilon}

\newcommand{\NN}{\mathbb{N}}

\newcommand{\RR}{\mathbb{R}}

\newcommand{\ep}{\epsilon}%
\newcommand{\fin}[1]{\mbox{Fin}(#1)}
\newcommand{\multi}[1]{\mbox{mul}(#1)}

\newcommand{\sympro}{\mbox{SP}}

\newcommand{\symproinf}{\mbox{SP}^\infty}

\newcommand{\loopsBM}{\Omega'_s BM}
\newcommand{\interval}{\mathcal I}
\newcommand{\intM}{\interval \times M}
\newcommand{\adm}[1]{W(#1)}
\newcommand{\adms}{W_s}
\newcommand{\ims}[1]{I_M(1,s)\times \{s\}}
\newcommand{\ems}[1]{E_M(1,s)\times \{s\}}
\newcommand{\openset}{O}
\newcommand{\filternumber}{l}
\newcommand{\elemone}[1]{E(#1)}
\newcommand{\elemtwo}[1]{F(#1)}
\newcommand{\pretensor}[1]{T_{#1}}
\newcommand{\translation}[1]{\tau_{#1}}
\newcommand{\intervalconvention}[3]{\,\ensuremath{_{#1}|{#2}|_{#3} }}
\newcommand{\Image}{\mathop{\mathrm{Im}}\nolimits}
\newcommand{\map}[1]{\mathop{\mathrm{Map}}\nolimits(#1)}
\title{Configuration space of intervals with partially summable labels}
\author{Shingo Okuyama and Kazuhisa Shimakawa}
\date{\today}
\begin{document}
\maketitle

\begin{abstract}
  A configuration space of intervals in $\RR^1$
with partially summable labels is 
constructed. It is a kind of an extension of the 
configuration space with partially summable labels
constructed by the second author and at the same time 
a generalization of the configuration space of intervals
with labels in a based space constructed by the first
author. An approximation theorem of the preceding
configuration space is generalized to our case. 
When partially summable labels are given by a partial abelian monoid $M,$
we prove that it is weakly homotopy equivalent to the space of based loops 
on the classifying space of $M$ under some assumptions on $M$.
\end{abstract}
\section{Introduction}
Configuration space is one of the most useful constructions in algebraic topology.
In particular, configuration space of unordered points in a topological space has been shown to
be a nice model for homotopical approximation to mapping spaces. 
It dates back to the works of Milgram\cite{milgram-iterated-loop-spaces}, May\cite{may-geometry-of-iterated},
and Segal\cite{segal-configuration-spaces-and},
which states that unordered points in $\RR^n$ with labels in a based space $X$ is weakly homotopy equivalent to $\Omega^n \Sigma^n X,$
if $X$ is path-connected. Then it is generalized to a setting with a manifold replaces 
$\RR^n,$ 
and with a partial monoid replaces a space $X$ as labels 
\cite{bodigheimer-stable-splittings-of}, \cite{macduff-configuration-spaces}, 
\cite{salvatore-configuration-spaces-with}, \cite{shimakawa-configuration-spaces-with}.
Based on these works and others, a new notion called factorization homology is 
developed recently \cite{ayala-francis-factorization-homology}(see
also \S 6 of \cite{knudsen-configuration-spaces-in}).

We consider a configuration space of intervals in $\RR$, in which two
intervals are pasted when meeting endpoints have opposite properties,
that is, one is open and the other is closed,
and a half-open interval annihilates when its length approaches zero.
These transformations of intervals are called cutting-pasting and
creation-annihilation, respectively.
If we attach labels on these intervals to consider a configuration 
space of intervals in $\RR$
with labels in some based space $X,$ then
we obtain a space which is 
weakly homotopy equivalent to $\Omega\Sigma X,$
the space of based loops in the reduced suspension of X \cite{okuyama-space-of-intervals}.
In this paper, we consider a configuration space of intervals with
labels in a partial abelian monoid $M.$
We can use partial sum to define a sort of interaction of intervals. 
So our configuration space has as its topology one which reflects this 
interaction as well as transformations
come from cutting-pasting and creation-annihilation mentioned above.
The purpose of this paper is to construct such a space
and show that there exists a weak homotopy
equivalence analogous to \cite{okuyama-space-of-intervals}. 
When we regard a based space $X$ as a topological partial abelian
monoid by a trivial sum, then the reduced suspension 
construction can be considered as a special case of
the classifying space construction
$BM$ of a partial abelian monoid $M$\cite{segal-configuration-spaces-and}. 

We can state our theorem as follows.
\begin{thm}\label{thm:main}
Let $M$ be a (discrete) partial abelian monoid whose elements are self-insummable.
Then the configuration space of intervals in $\RR$ with labels in $M$
is weakly homotopy equivalent to $\Omega B M.$
\end{thm}
Self-insummability of a partial abelian monoid which appears in the above theorem
is explained in \S 2. This condition is used in Lemma \ref{lem:unique-representation}.
This paper is organized as follows.
In \S 2.1, we give a definition of partial abelian monoid. In \S 2.2, we give a definition of a product of 
partial abelian monoids, which is shown to be useful for constructing configuration spaces. 
Configuration spaces of intervals $I$ and $I_M$ are defined in \S 3.1 and \S 3.2, respectively. 
To prove the weak equivalence, we construct a thickening $\widetilde{I}_M$ of $I_M$ and
maps
$$
 I_M \leftarrow \widetilde{I}_M \to \Omega BM,
$$
which induce isomorphisms on homotopy groups.
The thickening and other auxiliary spaces are defined in \S 3.3. 
In \S 4, we construct of a map $\alpha : \widetilde{I}_M \to \Omega B M.$
Then in \S 5, we give  a proof of the main theorem.

In this paper, intervals in $\RR^1$ with two different roles occur.
One is as an object of which configuration is to be considered and the other is as just 
a subset of $\RR^1$ for usual usage, that is, a domain of configuration, a domain of definition of a function, 
or an interval for a homotopy. For the latter, we use a standard notation such as
$ ( a , b ] = \{ x ~|~ a < x \leq b \}.$ For the former, we make some convention based on 
another standard notation such as $ ] a , b ] = \{ x ~|~ a < x \leq b \}.$ See \S 3.1 for the convention.
We work in the category of compactly generated weak Hausdorff spaces and
base points are assumed to be non-degenerate.

\section{Partial abelian monoid}

\subsection{Partial abelian monoid}

The notion of partial abelian monoid plays an important role in our work.
It is an abelian monoid with partially defined sum, whose definition is given
as follows : 

\begin{defn}\label{defn:pam}
A {\bf topological partial abelian monoid} is a space $M$ with base point $0$
equipped with a subspace
$M_2$ of $M\times M$ 
and a map $m : M_2\to M$ which satisfies
\begin{enumerate}
	\item $M\vee M\subset M_2,$ and $m(a,0)=m(0,a)=a$,	
	\item $(a,b)\in M_2$ if and only if $(b,a)\in M_2,$ and $m(a,b)=m(b,a),$
	\item $(m(a,b),c)\in M_2$ if and only if $(a,m(b,c))\in M_2,$ and $
	m(m(a,b),c)=m(a,m(b,c)) .$
\end{enumerate}
We denote $m(a,b)=a+b.$ By associativity, we can speak of
summable $n$-tuples in $M^n$, and we denote the subset of summable $n$-tuples by $M_n.$ 
\end{defn}

A partial abelian monoid $M$ is said to be self-insummable if every non-zero element
is self-insummable, that is, $\Delta \cap M_2 = \{ ( 0, 0 ) \}.$

Partial abelian monoid is also called partial commutative monoid. 
Abelian partial monoid defined in \cite{shimakawa-configuration-spaces-with} is the same concept in philosophy, 
but it is more general than the partial abelian monoid given above.

\subsection{A product of partial abelian monoids}

For a topological space $Y,$ let $\multi{Y}=\coprod_{n\geq 0}\sympro^n{Y}.$
It can be considered as the free abelian monoid generated by $Y_+=Y\cup \{0\}$ with an appropriate
topology, or equivalently,  as $\symproinf{Y_+}$, an infinite symmetric product introduced in 
\cite{dold-thom-quasifaserungen-und-unendliche}. 
In this paper, however, we often treat an element of $\multi{Y}$ 
as a finite multiset --- a finite ``set'' with 
repeated elements~\cite{stanley-enumerative-combinatorics}. The notions such as
cardinality, submultiset or union of multisets are defined so as to respect multiplicities
of elements. We denote by $ x_1\dotplus \dots \dotplus x_n$ to mean the multiset consisting of
elements $x_1,\dots, x_n$ where each element is repeated a number of times equal to
its multiplicity. The union of multisets $\alpha$ and $\beta$ is denoted by $\alpha \dotplus \beta.$
Thus if $\alpha = a_1\dotplus\dots \dotplus a_r$ and $\beta = b_1\dotplus \dots \dotplus b_s$ then
$\alpha \dotplus \beta = a_1\dotplus\dots\dotplus a_r\dotplus b_1\dotplus \dots\dotplus b_s.$
To give an element of $\multi{Y}$ is equivalent to giving a map $\sigma : S\to Y,$ 
where $S$ is a finite set. In this form, a submultiset of $\sigma$ can be given as a restriction
map $\sigma|T : T \to Y$ to a subset $T$ of $S.$ If $M$ is a partial abelian monoid 
and $\sigma = m_1\dotplus\dots\dotplus m_r \in \multi{M},$
we say that $\sigma$ is summable if $(m_1,\dots, m_r)$, in any order, is a summable $r$-tuple.
Otherwise, we say that $\sigma$ is insummable. On the other hand, we say that $\sigma$ is 
pairwise insummable if, for any subset $T\subset S$ of cardinality two, $\sigma|T$ is insummable.

Let $M$ and $N$ be partial abelian monoids. 
We denote by $p_1 : M\times N \to M$ and $p_2 : M\times N \to N$ the projections from $M\times N$
onto its first and second factors respectively.
Let $\sigma \in \multi{M\times N}$ and $\sigma : S \to M\times N$ be its representation as a map.
Consider the following property for $\sigma$ : for any subset $T$ of $S,$
if one of $p_i\circ (\sigma|T)$ is pairwise insummable then
the other is summable.
We denote by $\pretensor{M,N}$ the subspace of $\multi{M\times N}$ consisting of $\sigma$ with
this property. 

Let $\sim$ be the least equivalence 
relation on $\pretensor{M,N}$ which satisfies the following three conditions:
\begin{enumerate}
	\item[(R1)] If $m_1$ or $n_1$ is zero then 
	$$ (m_1,n_1)\dotplus\dots\dotplus(m_r,n_r) \sim (m_2,n_2)\dotplus\dots\dotplus (m_r,m_r),$$
	\item[(R2)] If $m_1 = m'_1 + m''_1$ then
	$$ (m_1,n_1)\dotplus\dots\dotplus (m_r , n_r ) 
			 \sim (m'_1,n_1)\dotplus(m''_1, n_1)\dotplus (m_2, n_2)\dotplus\dots \dotplus (m_r , n_r),$$
	\item[(R3)] If $n_1=n'_1 + n''_1$  then
	$$ (m_1,n_1)\dotplus\dots\dotplus (m_r , n_r )
			 \sim (m_1,n'_1)\dotplus (m_1, n''_1)\dotplus (m_2, n_2) \dotplus \dots \dotplus (m_r , n_r).$$
\end{enumerate}
Let $T_k$ be the subspace of $\pretensor{M, N}$ 
of elements of cardinality less than or equal to $k.$ Then the relation $\sim$ on
$\pretensor{M,N}$ induces a relation on $T_k.$
Then we denote by $M\otimes N$  the union $\bigcup_{k\geq 0} (T_k / \sim)$ 
of quotient spaces. Let $\pi_\otimes : \pretensor{M,N} \to M\otimes N$ be the natural map.
Two elements $[\alpha], [\beta]$ in $M\otimes N$ are summable if we can choose
their representatives $\alpha, \beta$ in $\pretensor{M,N}$ so that their sum $\alpha \dotplus \beta$ taken in $\multi{M\times N}$
is contained in $\pretensor{M,N}.$
Thus,  $M\otimes N$ is a partial abelian monoid in a natural way.

The product of partial abelian monoids defined above covers diverse examples.
\begin{examples*}
	A based space $X$ can be regarded as a trivial partial abelian monoid by setting
	$X_2 = X\vee X$ and $\mu : X\vee X \to X$ the folding map. Then 
	$X\otimes M$ is the configuration space of finite points in $X$ with labels in $M$
	such that only summable labels occur simultaneously.
\begin{enumerate}
	\item For two based spaces $X, X' , $
	their product $X\otimes X'$ coincides with their smash product
	$X \wedge X'.$
	\item \label{ex:classifying}Viewing $S^1$ as a based space, we get
	$S^1\otimes M = BM$ the classifying space defined in 
	\cite{segal-configuration-spaces-and}. In particular, if $M$ is a monoid this coincides with the
	McCord model of the classifying space\cite{mccord-classifying-spaces-and}.
	\item Let $X$ be a compact based space and $M = Gr := \sqcup Gr_n(\RR^\infty)$
	be the infinite Grassmannian with a partial sum defined only for two vector spaces
	which are perpendicular to each other. Then
	$X\otimes Gr = F(X)$ coincides with the configuration space stated in 
	\cite{segal-k-homology-theory}. This configuration space realizes connective $K$-homology.
	In \cite{tamaki-twisting-segal's-k-homology}, a similar construction is given, but it is enriched by an operad
	to make twisting on $K$-theory, thus larger than $X\otimes Gr.$
	
\end{enumerate}
	We denote by $\fin{Y}$ the set of finite subsets of a space $Y$ and
	give it a topology by
	$$\fin{Y} = \{\emptyset \} \sqcup Y \sqcup (Y^2\setminus\Delta)/\Sigma_2 %
	\sqcup (Y^3 \setminus \Delta_3)/\Sigma_3 \sqcup \cdots ,$$
	where $\Delta_n\subset Y^n$ is the fat diagonal.
	We regard $\fin{Y}$ as a partial abelian monoid by disjoint union of sets.
	We can realize configuration spaces in which non-summable labels can occur simultaneously, 
	by using $\fin{Y}.$
\begin{enumerate}
	\setcounter{enumi}{3}
	\item \label{ex:config} If $C_n = \fin{\RR^n}$ then
	$C_n \otimes X  = C_n(X)$
	is the configuration space of finite points in $\RR^n$ with labels in $X$
	\cite{segal-configuration-spaces-and}.
	\item \label{ex:shimakawa}$\fin{Y}\otimes M = C^M(Y)$ is  
	the configuration space of finite points in $Y$ with labels in $M$
	defined in \cite{shimakawa-configuration-spaces-with}. Note that in \cite{shimakawa-configuration-spaces-with}, the notion of
	abelian partial monoid is more general concept than our notion of partial
	abelian monoid.
\end{enumerate}
	Among examples in an extreme case, when $M$ or $N$ is a monoid, we have:
\begin{enumerate}
	\setcounter{enumi}{5}
	\item $ X\otimes \NN = \symproinf{X}$, the infinite symmetric product on a based space $X$ introduced
	in \cite{dold-thom-quasifaserungen-und-unendliche}.
	\item For abelian groups $A, B, $ their product $A\otimes B$ defined here
	is the usual tensor product of modules.
\end{enumerate}	
\end{examples*}
\section{Configuration space of intervals}
\subsection{Configuration space of intervals without labels}

Let $H$ be the half-plane in $\RR^2$ given by
$ H = \{ (u, v) \in \RR^2 ~|~  u \leq v\}$
and $P = \{ \pm 1\}$ be the set of ``parities'' considered as a space with discrete
topology. To any point $( u , v ; p , q )$ in the direct product $H \times P^2$
with $u < v,$
we assign an interval
$$
J = \{ x \in \RR ~|~ u <_p x <_q v\} 
$$
in $\RR,$ where the symbol
``$<_p$'' denotes the inequality ``$\leq$'' if  $p = +1$ and ``$<$'' if $p = -1.$
We also denote $J = \intervalconvention{p}{a , b}{q}.$
If $u = v,$ it is not an interval anymore,  but we allow it only if
$p\neq q$ and call it a degenerate interval. For notational convention, let $\bar{p} = -p$ for any $p \in P.$
Let 
$$
\interval = \{ (u , v ; p, q ) \in H\times P^2~|~u \leq v , ~p\neq q \mbox{~if~}u = v\}
$$
and we identify the element of $\interval$ with the interval assigned to it.
Let $u_L$ and $u_R$ denote the projections $\interval \to \RR$ onto its first and 
second factors respectively, and
$p_L$ and $p_R$ denote the projections onto its third and fourth projections $\interval \to P$ respectively.
Under the identification stated above, $u_L(J)$ or $u_R(J)$ is the coordinate of the left or right
endpoint of $J$, respectively, while $p_L(J)$ or $p_R(J)$ is the parity of the left or right endpoint, 
respectively.
For any two elements $J_1, J_2 \in \interval,$
we denote $J_1 \leq J_2$ if either (1) $u_R( J_1 ) < u_L(J_2)$ or (2) $u_R(J_1) \leq u_L(J_2)$ and $p_R(J_1) \neq p_L(J_2).$
We denote $J_1 < J_2$ if $u_R(J_1) < u_L(J_2).$ 

Let $L_r$ be the subspace of $\interval^r$ given by
$$
L_r = \left\{ ( J_1 , \dots , J_r ) \in \interval^r~\left|~\begin{array}{l}
J_1\leq \dots \leq J_r
\end{array}\right\}\right. _,
$$
Then $L_r$ is the configuration space of $r$ bounded intervals in $\RR$
such that endpoints of the same parity do not collide, with additional limiting
points of creation and annihilation provided by degenerate intervals. 
Later, these additional points are identified with the
empty configuration. Meanwhile, if endpoints of two distinct intervals, one open and the other closed,
collide then two meeting intervals are pasted and the configuration of these two intervals
is identified with the configuration of one interval. 
To establish these identifications, 
let $\sim$ be the least equivalence relation on $\bigsqcup_{r\geq 0} L_r$ which satisfies
the following two relations :
\begin{enumerate}
	\item $( J_1, \dots, J_r ) \sim (J_1, \dots, J_{i-1}, J_{i+1}, \dots , J_r ) $ if $u_L(J_i) = u_R(J_i).$
	\item $( J_1, \dots, J_r ) \sim (J_1, \dots, J_{i-1}, K , J_{i+2} ,\dots , J_r ) $ if $u_R(J_i) = u_L(J_{i+1}),$
	where $K = ( u_L( J_i ) , u_R(J_{i+1}), p_L( J_i ), p_R( J_{i+1}) ).$
\end{enumerate}
Let $I$ be the quotient space of $\bigsqcup_{r\geq 0} L_r$ by this equivalence relation.
The image of $L_0$ is a single point and denoted by $\emptyset,$ which is
treated as the base point of $I.$
Let $L'_r$ be the subspace of $L_r$ given by
$$
L'_r = \left\{ ( J_1 , \dots , J_r ) \in \interval^r~\left|~\begin{array}{l}
J_1< \dots < J_r, \\
u_L(J_i) < u_R(J_{i}) \mbox{~for each~} i = 1,\dots ,r
\end{array}\right\}\right. _.
$$
As sets there is a bijective correspondence between
$\bigsqcup_{r\geq 0} L'_r$ and $I.$ We call 
an element in $L'_r$ corresponding to an element $\xi$ of $I$
the {\bf reduced representative} of $\xi.$
Let $\pi_\interval$ denote the composite $\interval \twoheadrightarrow L_1 \hookrightarrow I.$ 
\subsection{Configuration space of intervals with partially summable labels}

Let $U = (a, b)$ be an open interval in $\RR.$ We consider two special
types of elements in $\multi{\interval \times M}.$
\begin{enumerate}
	\item[(E1)] $e  =  ( J , n ) $ with one of the following :
	\begin{enumerate}
		\item $J = ] a, b [$
		\item $J = ] a, w [$ or $J = ] a, w ],~ a < w < b$ 
		\item $J = ] w, b [$ or $J = [ w, b [,~ a < w < b$ 
		\item $J = \intervalconvention{p}{w_1, w_1}{q}  ,~ a < w_1 < w_2 < b, p+q = 0$ 
	\end{enumerate}
	\item[(E2)] $e =  ( J_1 , n ) \dotplus ( J_2 , n )$ with 
			$J_1  = \left] a , w_1 \right|_p , J_2 =\,_q| w_2 , b [ $ and $a\leq w_1 < w_2 < b$, $p+q=0$ 
\end{enumerate}
where $n$ is a non-zero element in $M,$ which is denoted by $n(e).$
We call such $e$ {\bf an  elementary configuration in $U$.}

Let $ ( J_1, m_1 ) \dotplus \dots \dotplus ( J_r, m_r ) $ be a representative of $\xi \in \multi{\intM}$.
Suppose that $i_1, \dots , i_\lambda$ are all the subscripts $i$ such that $J_i \cap U \neq \emptyset.$
Then let $\xi|_U$ denote the multiset in $\multi{\intM}$ represented by
$ ( J_{i_1}\cap U , m_{i_1} )\dotplus \dots\dotplus (J_{i_\lambda}\cap U , m_{i_\lambda}) ).$

Let $\pi_{mul} : \multi{\intM} \to \multi{I\times M}$ be the map induced by 
$\pi_\interval\times id_M : \intM \to I\times M$ and
let $\pretensor{\interval , M} = \pi_{mul}^{-1}(\pretensor{I, M}) \subset \multi{\interval\times M}.$
An element $\xi \in \pretensor{\interval, M}$ is said to be
{\bf admissible} if for any $t\in \RR$ there exists an open interval $U= (a,b)$ which contains
$t$ such that $\xi|_U$ is represented by $e_1 \dot{+} \dots \dot{+} e_r$ for some elementary
configurations $e_1, \dots , e_r$ in $U$ such that $(n(e_1), \dots , n(e_r)) \in M_r.$
In this case, such a representative
$e_1 \dot{+} \dots \dot{+} e_r$ is called an admissible sum of elementary configurations.
If, moreover, there exist $\ep>0$ and an interval $U$ can be taken as $U = (t-\ep, t+\ep)$ for all $t$,
then we say that $\xi$ is {\bf $\ep$-admissible}.
Let $V = ( a, b )$ be an open interval with $b-a > \ep.$ We say that an $\ep$-admissible 
element $\xi$ is supported by $V$ if $\xi |_{(a+\ep/2, b-\ep/2)} = \xi.$ 
Let $W, W(\ep)$, and $W(\ep, V)$ be the subspace of $\pretensor{\interval, M}$ which consists of 
admissible elements, $\ep$-admissible elements, and $\ep$-admissible elements supported by $V$, 
respectively. 
If $\ep > \ep'$ and $V\subset V',$ then we have a natural inclusion $W(\ep, V) \subset W(\ep', V').$
Let $I_M$ be the image in $I\otimes M$ of $W$
under the natural map $\pi_\otimes \circ \pi_{mul}.$ 
Let also $I_M(\ep)$ and $I_M(\ep, V)$ be the image in 
$I\otimes M$ of $W(\ep)$ and $W(\ep, V)$, respectively , 
 under $\pi_\otimes\circ \pi_{mul}.$ 
We call
$$I_M = \bigcup_{\ep>0, V} I_M(\ep, V).$$
a {\bf configuration space of intervals with partially summable labels,}
where its topology is the weak topology of the union.

\subsection{Thickening, Moore type variant, and the total space}
If $V = ( 0 ,s )$ then let $I_M( \ep , s ) = I_M ( \ep , V )$ and $W( \ep, s ) = W( \ep, V ).$
We define 
\[
\widetilde{I}_M = \bigcup_{\ep > 0, s \geq 0} I_M(\ep, s ) \times \{\ep\}\times\{s\}
\]
and give it the topology as a subspace of $I_M \times \RR^2.$

If $s \leq \ep,$
$I_M(\ep, s )$ consists of one point, the element $\emptyset$ in $I_M$ which represents
the empty configuration.  As a base point of $\widetilde{I}_M,$ we take
$(\emptyset, 1,0).$

\begin{prop}
	The projection $p : \widetilde{I}_M\to I_M$ onto the first component is a weak homotopy equivalence.
\end{prop}

\begin{proof}
	Let $f : S^n \to I_M$ be a map which represents an element $\alpha \in \pi_n(I_M).$ 
	The image of $f$ is a compact (Hausdorff) subspace of $I_M,$
	since $S^n$ is compact and all our spaces are weak Hausdorff and compactly generated 
	\cite{mccord-classifying-spaces-and}.
	Then we can find $\ep >0$ and $s>0$ such that $I_M(\ep , s )$ contains the image of $f$ (
	weak Hausdorff case of Lemma 9.3 in \cite{steenrod-convenient-category-of}).
	However, we have a homotopy equivalence
	$I_M(\ep ,s ) = I_M(\ep, s ) \times \{\ep\} \times \{s\}
	\hookrightarrow \widetilde{I}_M.$
	Composing this inclusion, we get a map
	$f' : S^n \to \widetilde{I}_M$ which maps to $\alpha$ under $p_*.$	
	 This proves that $p_* : \pi_n(\widetilde{I}_M)\to \pi(I_M)$	is surjective. 
	 
	 For injectivity, suppose $\alpha = [f] \in \pi_n(\widetilde{I}_M)$ maps to  
	 $\pi_*(\alpha) = [g] \in\pi_n(I_M),$ which is zero. Let $H : S^n\times I \to I_M$ be a homotopy 
	 between $g$ and $*.$ Then by the compactness, we
	can assume that the image of $H$ is contained in  $I_M(\ep , s )$ for some $\ep $ and $s.$
	It is similar to the previous case, to show that there exists a map $H' : S^n\times I \to \widetilde{I}_M.$
	$H'$ gives a null homotopy of some map $f' : S^n \to \widetilde{I}_M,$ which is homotopic to $f.$ 
\end{proof}

We proceed to define $E_M.$
For any $J = \intervalconvention{p}{u, v}{q}$ in $\interval,$ we define its ``mirror image'' $J^\mu$ as
$J^\mu = \intervalconvention{\bar{q}}{-v,-u}{\bar{p}}.$ This defines
an involution on $\interval,$ which induces an involution on $I_M,$ which is also denoted by $\mu.$
The space of fixed points $E_M = I_M^\mu$ can be considered as the space of 
intervals in $[0, \infty )$ such that $0\in \RR$ works as a ``vanishing point.''
For $s > 0,$ let $E_M(\ep, s) $ be the space $E_M(\ep, s) = E_M \cap I_M(\ep, (-s, s)).$
Then we define 
\[
\widetilde{E}_M = \bigcup_{\ep > 0, s \geq \ep} E_M(\ep, s ) \times \{\ep\}\times\{s\} \subset E_M \times \RR^2
\]
and give it the relative topology.
If $s = \ep,$
$E_M(\ep, \ep )$ contains only one point, the element $\emptyset$ in $E_M$ which represents
the empty configuration.  As a base point of $\widetilde{E}_M,$ we take
$(\emptyset, 1,1).$

\begin{prop}
	$\widetilde{E}_M$ is contractible.
\end{prop}

\begin{proof}
	When $h : \RR\to \RR$ is some increasing function, we put $h(J) = \intervalconvention{p}{h(u) , h(v)}{q}$
	for any interval $J = \intervalconvention{p}{u, v}{q}\in \interval.$
	If $h(u) = h(v)$ then $h(J)\in \interval$ represents a (degenerate) interval only if $p\neq q.$
	Let $h_t : ( -s, s ) \to ( -s, s )  ~(0 \leq t\leq 1)$ be the homotopy defined by
	$$
	h_t( u ) = \left\{
	\begin{array}{cc}
		u- ts & u\geq ts\\
		0 & |u| < ts\\
		u + ts & u\leq -ts
	\end{array}
	\right.
	$$
	Then we define a homotopy $H_{t} : E_M(\ep, s) \to E_M(\ep, s)$ as follows. If
	$\xi =  (J_1, m_1)\dotplus \dots\dotplus (J_r, m_r)   \in \multi{\interval \times M}$ is a 
	$\mu$-invariant representative
	of $\xi \in E_M(\ep, s)$ then $H_t(\xi)$ be the element of $E_M(\ep,s)$ represented by
	$$
		H_t(\xi) = (h_t(J_1), m_1)\dotplus \dots\dotplus (h_t(J_r) , m_r).
	$$
	For each $i,$ $h_t(J_i)$ is at least a degenerate interval since our representative is $\mu$-invariant.	
	It is easy to see that $H_t$ defines a contraction of $\widetilde{E}_M.$
\end{proof}

We define $d : \multi{\interval \times M} \to  \multi{\interval \times M}$ by
$$d(  ( J_\lambda, m_\lambda )_\lambda ) = ( J_\lambda, m_\lambda )_\lambda \dotplus   ( J^\mu_\lambda, m_\lambda )_\lambda. $$
By restricting $d$ to $\adm{\ep, s},$ we define an embedding $I_M(\ep,s) \to E_M(\ep, s)$ then an embedding $\widetilde{I}_M \to \widetilde{E}_M.$
In the following sections, we only give definitions and proofs for the case $\ep = 1.$
Required changes for other cases are obvious. 
\section{Construction of $\alpha : \widetilde{I}_M \to \Omega' BM$}
We are going to define a map $\alpha : \widetilde{I}_M \to \Omega' BM$
 in the following steps:
Let $U_t = ( t - 1,  t + 1 )$ and $V_t = ( t -1/2, t + 1/2 )$ and
$\adms = \adm{1,s} \times \{s\}.$
\begin{enumerate}
	\item We can define a map $\omega(J) : V_t\to S^1$ for each interval $J\in (H\cap U^2_t)\times P^2.$
	Then we define
	$$\omega_t : \adms \to \multi{\map{V_t, S^1} \times M}.$$
	Let $G$ denote the image of $\adms$ under $\omega_t.$

	 \item  \label{stepmu}
	 We can define a map $s_t : G \to \multi{\map{V_t, S^1} \times M}.$
	 We denote its image by $G'.$
	\item $G'$ is mapped into $\map{V_t, \pretensor{S^1, M}}$ under the sequence of natural maps
	\begin{eqnarray*}
		\multi{\map{V_t, S^1} \times M} & \to & \multi{\map{V_t, S^1\times M} } \\
												& \to & \map{V_t, \multi{S^1 \times M} } .  
	\end{eqnarray*}
	
	\item Composition of maps given in (1)$\sim$(3) and the map induced by
	the quotient $\pretensor{S^1, M} \to S^1\otimes M = BM$
	gives us a map 
	$\alpha_t : \adms \to \map{V_t, BM}.$ 
	
	\item 	$\alpha_t ( [\xi])  \in \map{ V_t, \pretensor{S^1, M}}$
	for two distinct $t\in [0, s]$ coincide on their intersection of the domain of definition, so we
	can achieve extension to get an element in
	$\Omega'_s\pretensor{S^1, M}.$ Thus we get a map	
	$$  \adms  \to \Omega'_s BM$$
	
	\item We can take quotient on the source to define a map
	$\alpha(\ep, s) :  I_M(\ep, s) \to \loopsBM.$
	\item We consider all the $s > 0$ and 	
	take a union	to get a continuous map $\alpha : \widetilde{I}_M \to \Omega' BM.$
\end{enumerate}
Continuity of the maps are clear from its construction.

\begin{center}
\begin{tikzcd}
		\adms \ar{r}{\omega_t}\ar[end anchor = north west]{rdddd}[left]{\alpha_t} %
																										& G \ar{d}{s_t}%
													\ar[hook]{r} & \multi{\map{V_t, S^1} \times M}
\\
												& G' \ar[hook]{r}\ar{dd}  & \multi{\map{V_t, S^1} \times M}\ar{d}\\
												 &				&	\multi{\map{V_t, S^1\times M} }\ar{d} \\
												 &	\map{V_t, \pretensor{S^1, M} }\ar[hook]{r}\ar{d} 	&	\map{V_t, \multi{S^1\times M} }\\
												 & \map{V_t, BM }.\\
\end{tikzcd}
\end{center}
\subsection{A map $\omega_t : \adms \to \multi{\map{V_t, S^1}\times M}$}
Let $\mathcal K$ be the subspace of  $\interval $ defined by
$$\mathcal K = \{ J = \intervalconvention{p}{ u, v}{q} ~|~ v-u > 1 \mbox{~if~} p=q\}$$ 
For any $J =  \intervalconvention{p}{ u, v}{q} \in\mathcal  K,$ let $\Omega(J) : \RR \to S^1$ be 
a map defined as follows:
\begin{enumerate}
	\item if $v - u > 1$ then
		$$\omega (J)(s) = \left\{ 
			\begin{array}{ll}
				p(s-u-1/2) & (\mbox{if~} u-1/2 < s \leq u+1/2)\\
				0 & (\mbox{if~} u+1/2 < s \leq v-1/2)\\
				q(s-v+1/2) & (\mbox{if~} v - 1/2 < s \leq v + 1/2)\\
				1 & (\mbox{otherwise})
			\end{array}
		\right.$$
	\item if $v - u \leq 1$ then (in this case $p+q = 0$)
		$$\omega (J)(s) = \left\{ 
			\begin{array}{ll}
				p(s-u-1/2) & (\mbox{if~} u - 1/2 < s \leq v - 1/2)\\
				p(v-u-1) & (\mbox{if~} v - 1/2 < s \leq u + 1/2)\\
				q(s-v+1/2) & (\mbox{if~} u + 1/2 < s \leq v+1/2)\\
				1 & (\mbox{otherwise})
			\end{array}
		\right.$$
\end{enumerate}
The correspondence $J\mapsto \omega(J)$ gives us map $\omega :\mathcal  K \to \map{\RR^1, S^1}.$
We are now going to define a map $\omega_t : W_s \to \multi{\map{V_t, S^1}\times M}.$
\begin{enumerate}
\item If $e =  ( J , m ) $ is an elementary configuration in $(t-1,t+1)$ of type (E1),
\begin{enumerate}
	\item we put $\tilde{J} = J$
	\item if $J = ] a, w [$  then we put $\tilde{J} = [a, w[,$ otherwise $\tilde{J} = J,$
	\item if $J = ] w, b [$  then we put $\tilde{J} = ]w, b]$ otherwise $\tilde{J} = J,$
	\item we put $\tilde{J} = J$
\end{enumerate}
We put $\tilde{e} =  ( \tilde{J} , m )  \in \multi{\interval\times M}.$ 
\item If $f =  (K, n)\dotplus (K',n) $ is an elementary configuration in $]t-1,t+1[$ of type (E2),
\begin{enumerate}
	\item $\tilde{K} = K,$ and if $K' = ] w_2, b [$  then we replace it with $\tilde{K'} = ] w_2, b],$ otherwise $\tilde{K'} = K',$
	\item $\tilde{K'} = K,$ and if $K = ] a,  w_1 [$  then we replace it with $\tilde{K} = [a , w_1[,$ otherwise $\tilde{K} = K,$
\end{enumerate}
\end{enumerate}
We put $\tilde{f} =  ( \tilde{K} , n )\dotplus (\tilde{K'},n)  \in \multi{\interval\times M}.$ 
Resulting element $\tilde{e} \in \multi{\interval\times M}$ is called the replacement of $e.$
For $\xi = e_1 \dotplus \dots \dotplus e_r \dotplus f_1 \dotplus \dots \dotplus f_s,$
let $r(\xi ) = \tilde{e}_1 \dotplus \dots \dotplus \tilde{e}_r \dotplus \tilde{f}_1 \dotplus \dots \dotplus \tilde{f}_s.$

We can define
$\omega_t : \adms \to \multi{\map{V_t, S^1}\times M}$
by the following diagram:
\begin{center}
\begin{tikzcd}
	\adms \ar{r}{r} \ar{rrd}[below]{\omega_t}&
	\multi{\mathcal K\times M} \ar{r} & 
	\multi{\map{\RR^1, S^1}\times M} \ar{d}{\mbox{res}}\\
	 & & \multi{\map{V_t, S^1}}_,
\end{tikzcd}
\end{center}
where the second map in the upper horizontal line is $\multi{\omega\times id_M}.$
This map is essentially the one defined in \cite{okuyama-space-of-intervals} for the 
case of the space of intervals with labels in a trivial partial abelian monoid.
Let $G$ denote the image of $\omega_t.$
\subsection{A map $s_t : G \to \multi{\map{V_t, S^1} \times M}$} 

We can define a `sum' $s_t : G \to \multi{\map{V_t, S^1} \times M}.$
For any element $\xi\in \adms,$ $\xi_t = \xi|_{U_t}$ is of the form
	$$\xi_t = e_1 \dotplus \dots \dotplus e_r \dotplus f_1 \dotplus \dots \dotplus f_s,$$
	where $e_i$ is an elementary configuration of type (E1) for each
	$1\leq i\leq r$ and $f_j$ is that of type (E2) for each $1\leq j \leq s.$ 
	Let $e_i =  ( J_i, m_i ) $ and $f_j =  ( K_j, n_j )\dotplus ( K'_j, n_j )$ so that 
	each $J_i$ is one of (F0)$\sim$ (F3) 
	and each $K_i$ is (F1) and each $K'_i$ is (F2) shown below.
\begin{enumerate}
	\item[(F0)] $J = ] t-1, t+1 [,$
	\item[(F1)] $J = ] t-1, w [$ or $J = ] t-1 , w ]$, where $t-1 < w < t+1,$
	\item[(F2)] $J = ] w, t+1 [$ or $J = [ w, t+1 [$, where $t-1 < e < t+1,$
	\item[(F3)] $J = ] w_1 , w_2 ]$ or $ J = [ w_1 , w_2 [,$ where $t-1 < w_1 \leq w_2 < t+1.$
\end{enumerate}
Accordingly, $\omega_t(\xi_t) \in G$ is represented by
$$
 (\varphi_1, m_1) \dotplus\dots\dotplus(\varphi_r , m_r) 
\dotplus
 (\psi_1, n_1) \dotplus (\psi', n_1 )\dotplus\dots \dotplus (\psi_s, n_s) \dotplus (\psi'_s, n_s),
$$
where $\varphi_i = \omega(J_i)|_{V_t}$ and $\psi_j = \omega(K_j)|_{V_t}, \psi'_j = \omega(K'_j)|_{V_t}.$

\begin{lem}\label{lem:unique-representation}
Suppose that every non-zero element of $M$ is self-insummable.
Then any element of $G$ is of the form
$$
 (\varphi_1, m_1) \dotplus\dots\dotplus(\varphi_r , m_r) 
\dotplus
 (\psi_1, n_1)\dotplus (\psi', n_1 )\dotplus\dots \dotplus (\psi_s, n_s) \dotplus (\psi'_s, n_s) ,
$$
where $\varphi_i = \omega(J_i)$ is an image of an element of type (E1) 
and $\psi_j = \omega(K_j), \psi'_j = \omega(K'_j)$ constitute a pair of an image
of an element of type (E2).
This representation is unique up to order.
\end{lem}
\begin{proof}
	Let $m_1\dotplus \dots \dotplus m_r \dotplus n_1\dotplus n_1\dotplus \dots \dotplus n_s\dotplus n_s$ be an element of
	$\multi{M}$ such that $(m_1, \dots ,m_r , n_1 ,\dots ,n_s) \in M_{r+s}.$ Since, by assumption, $(n_i , n_i ) \not\in M_2,$ it
	is easy to see that this representation is unique.
\end{proof}

Let $ (\psi, n )\dotplus ( \psi', n ) \in \multi{\map{V_t, S^1}\times M}$ be the image of an
elementary configuration $(K, n)\dotplus (K', n) $ of type (E2).
Let 
\[ 
K = \left] t-1 , v \right|_{q} \mbox{~and~} K' = \,_{\bar{q}}| u', t+1 [.
\]
Then we define
$\psi \oplus \psi' : V_t \to S^1$
by
\[
	(\psi \oplus \psi') (s) = %
	\left\{\begin{array}{ll}
				\psi(s)  & (\mbox{if~} s \leq u'-1/2)\\
				q (u'-v) & (\mbox{if~} u'- 1/2 \leq s < v+1/2)\\
				\psi'(s) &  (\mbox{if~} s > v_1+1/2).\\
			\end{array}\right. 
\]
Now a map $s_t : G \to \multi{ \map{ V_t, S^1} \times M   }$ is defined by
\begin{eqnarray*}
s_t ( \omega (\xi_t ) ) & = & \sum_{i}   (\omega(J_i), m_i) 
			  \dotplus \sum_{j}  (\omega(K_j)\oplus \omega(K'_j), n_j) .
\end{eqnarray*}

	It is clear that $s_t : G \to \multi{ \map{ V_t, S^1} \times M}$
	is continuous.
	 
\subsection{Final steps}
	It is readily seen that
	$G'$ is mapped into $\map{V_t, \pretensor{S^1, M}}$ under the sequence of natural maps
	\begin{eqnarray*}
		\multi{\map{V_t, S^1} \times M} & \to & \multi{\map{V_t, S^1\times M} } \\
												& \to & \map{V_t, \multi{S^1 \times M} } .  
	\end{eqnarray*}
Thus $\mu\circ\omega_t $ followed by this sequence gives a map
$ \adms \to \map{ V_t, \pretensor{S^1, M} }.$
Next, we	take the quotient $\pretensor{S^1, M} \to S^1\otimes M = BM,$ then maps altogether 
gives us a map $\alpha_t : \adms \to \map{V_t, BM}.$ 
Let $[\xi] \in \adms,$ then
	$\alpha_t ( [\xi ])  \in \map{ V_t, BM}$
	for two distinct $t\in [0, s]$ coincide on their intersection of the domain of definition and
	we can take the union of them to get an element in
	$\Omega'_s BM.$ Thus we get a map	
	$  \adms  \to \loopsBM.$

	However, this map factors through the natural quotient map
	$\adms \to \ims{s}$ as the following diagram, 
	\begin{center}
	\begin{tikzcd}
		\adms \ar{r} \ar{d}& \Omega'_s BM\\
		 \ims{s}\ar{ru}[below right]{\alpha(s)}
	\end{tikzcd}
	\end{center}
	so we have a map $\alpha(s) :  \ims{s} \to \loopsBM.$
	
	Finally, we consider all the $s > 0$ and 	
	take a union	to get a continuous map $\alpha : \widetilde{I}_M \to \Omega' BM.$
	
	Let $P' X$ denote the space of Moore paths on $X$;
\[
	P'X = \{ f : [ 0, s ] \to X ~|~ \mbox{continuous} , f(s) = *\}.
\]
Then a map $\beta : \widetilde{E}_M \to P'BM$ is defined by using $\alpha^\ep_{(-s,s)}.$
A map $p : \widetilde{E}_M \to BM$ is defined by the composite
$\widetilde{E}_M \to P'BM \to BM,$
where $P'BM \to BM$ is the evaluation at $0.$
	\begin{center}
	\begin{tikzcd}
		\widetilde{E}_M \ar{r}{\beta} \ar{rd}[below]{p} & P' BM \ar{d}{ev_0}\\
		& BM.
	\end{tikzcd}
	\end{center}

\section{Proof of the main theorem}
Proof of the main theorem is based on the following proposition.
\begin{prop}\label{prop:quasifibration}
Let $M$  be a (discrete) self-insummable partial abelian monoid.
The map $p : \widetilde{E}_M \to BM$ is a quasifibration with fiber $\widetilde{I}_M.$ 
\end{prop}
Assuming the above proposition, we can give a
\begin{proof}[Proof of Theorem 1]
In the following diagram,
\begin{center}
\begin{tikzcd}
		\widetilde{I}_M \ar{r}\ar{d} & \widetilde{E}_M\ar{r}\ar{d}  & BM \ar{d}\\
		\Omega'BM \ar{r}& P' BM \ar{r}& BM\\
\end{tikzcd}
\end{center}
lower horizontal line is the path-loop fibration and
vertical map on the right is identity map and vertical map in the middle is 
a weak homotopy equivalence, since it is 
a map between weakly-contractible spaces, 
hence so is the vertical map on the left.
\end{proof}
This section is devoted to giving a proof of
Proposition \ref{prop:quasifibration}.
\subsection{Dold-Thom criterion (1)}
For any element $\xi$ of $E_M,$ we can find a representative of the form 
$$S_- \dotplus S_0 \dotplus S_+\in 
\multi{\interval\times M},$$
which satisfies the following conditions:
\begin{enumerate}
	\item $S_+|_{( 0,\infty )} = S_+$
	\item $(S_-)^\mu = S_+$
	\item $S_0$ is a multiset of the following form
	$$ (J_1, m_1)\dotplus \dots\dotplus (J_r, m_r)$$
	where $J_i = \intervalconvention{\bar{p_i}}{-u_i, u_i}{p_i}$ for some $u_i >0$ and $p_i \in P$ for
	each $ i= 1, \dots , r .$
\end{enumerate}
We can also take such a representative in a reduced form.
Then from $S_0$ we construct a new multiset
\[
	S'_0 =  ( K_1, m_1)\dotplus \dots\dotplus (K_r, m_r) 
\]
with 
$K_i = [ 0 , u_i |_{p_i}.$ We call the element 
$S'_0 \dotplus S_+ \in \multi{\interval \times M}$
 {\bf a representative of the positive part} of $\xi.$

In the following, we use the standard filtration $F_l BM~(l\geq 0)$ of 
$BM = S^1 \otimes M,$ appeared in \S 2. 
\begin{lem}\label{lem:key-lemma1}
	Let $l \geq 1.$
	If we denote $V = F_l BM \setminus F_{l-1} BM,$ then there exists a homotopy
	equivalence $p^{-1} V \simeq V\times \widetilde{I}_M.$ 
	Accordingly, $V$ is a distinguished set with respect to $p.$
\end{lem}

\begin{proof}
 We define a map $\varphi : p^{-1}V\to \widetilde{I}_M$ by
taking the positive part and attach a ``cap'' on the left.
The cap consists of 
left-open intervals of length 1, which is contained in
 $( 0, 2 )$. To state the construction precisely,
let $ (p^{-1} V)_s$ denote the subspace
$p^{-1} V \cap \ems{s}$ of $p^{-1}V.$
We define a map
$\varphi_s : (p^{-1}V)_s \to \ims{s+2}$ as follows.
Let $(\xi, s ) \in (p^{-1}V)_s$ and let
$S'_0 \dotplus S_+$ be a representative of the positive part of $\xi.$
Among all the intervals occurring in this representation, we collect
ones which have $u_L(J) \leq 1/2.$ There should be exactly
$j$ of such intervals in number, so we may assume that they constitute 
a multiset
$$
(J_1, m_1) \dotplus \dots\dotplus (J_j, m_j).
$$
Then the ``cap'' of $S'_0 \dotplus S_+$ is the multiset 
$$
C =  (K_1, m_1)\dotplus \dots \dotplus (K_j, m_j) 
$$
where 
$K_i = \left] 1 - u_L(J_i), 2 - u_L(J_i) \right|_{p}, p = \overline{p_L(J_i)}.$
Now, $\varphi_s ( \xi )$ be the element of $\ims{s+2}$ represented by
$$
C \dotplus \translation{2}(S'_0 \dotplus S_+ ),
$$
where $\translation{2}$ denotes the translation by $2.$
Now
$\varphi : p^{-1} V \to V\times \widetilde{I}_M $ is defined by
$(\xi, s) \mapsto (\varphi_s(\xi) , s+2 ).$
Thus we have defined a map
\[
 ( p, \varphi ) : p^{-1} V \to V\times \widetilde{I}_M.
\]
We define a map $\psi : V\times \widetilde{I}_M \to p^{-1}V$
by inserting ``the standard lift of the $V$-component'' to the $\widetilde{I}_M$-component.
The standard lift of $v\in V$ is an element of
$\ems{2}$ which maps to $v$ under $p,$ whose 
positive part consists of right closed intervals.
To state the construction precisely,
we define a map $\psi_s : V\times \ims{s} \to (p^{-1}V)_{s+2}.$
For any element $v = t_1\otimes m_1 + \dots + t_j\otimes m_j \in V, $ we set
\[ L_i =  \,_{p_i}| |t_i|/2,  |t_i|/2 +1 ] \in \interval \]
where $p_i = -t_i/|t_i|$ if $t_i \neq 0$ and
$p_i$ is any of $\pm 1$ if $t_i = 0.$ Then ``the standard lift'' of $v$ is the multiset
$$
L =   ( L_1 , m_1 )\dotplus \dots\dotplus ( L_j , m_j ) .
$$
Then we define
$\psi_s(v, \xi)$ to be the element of  $(p^{-1}V)_{s+2}$ represented by
the multiset 
$   L \dotplus \translation{2}(X)  \dotplus \mu(L \dotplus \translation{2}(X)).$
Now 
$\psi : V\times \widetilde{I}_M \to p^{-1}V$ is defined by
$( v, (\xi, s)) \mapsto (\psi_s(v, \xi) , s+2).$

Now, we define a homotopy $H' : p^{-1}V \to p^{-1}V $ between $\psi \circ ( p,  \varphi ) $ and $id_{p^{-1}V}.$
We can do this by pushing everything to $2\in \RR$ and erases the part of the configuration on the left of $2.$
For this purpose, let $h'_t : \RR\to \RR ~(0\leq t\leq 1)$ be the homotopy defined by
	$$
		h'_t( u ) = \left\{
		\begin{array}{cc}
			u- 2 t & u\geq 2 t\\
			0 & |u| < 2t\\
			u + 2 t & u\leq -2 t
		\end{array}
		\right.
	$$
	Then we define a homotopy $H'_{t} :  (p^{-1}V)_s\to (p^{-1}V)_s$ as follows. If
	$\xi =  (J_1, m_1)\dotplus \dots\dotplus (J_r, m_r)  \in \multi{\interval \times M}$ is a 
	representative of 
	$\xi \in \Image [ \psi \circ ( p, \varphi ) ]$
	then let $H'_t(\xi)$ be the element of $(p^{-1}V)_s$ represented by
	$$
		(\translation{2}h'_t \translation{-2}(J_1), m_1)\dotplus \dots\dotplus (\translation{2}h'_t \translation{-2}(J_r) , m_r).
	$$
	For each $i,$ $h_t(J_i)$ is at least a degenerate interval by our construction.
	It is easy to see that $H'_t$ defines a homotopy $H' : p^{-1}V\to p^{-1}V$ between 
	$\psi \circ ( p, \varphi ) $	and $id_{p^{-1}V}.$

	Next, using the same homotopy $h'_t : \RR\to \RR~(0\leq t\leq 1),$ we define a homotopy
	$H''_t : V\times \widetilde{I}_M\to V\times \widetilde{I}_M$ between 
	$ (p, \varphi)\circ \psi $ and $id_{V\times \widetilde{I}_M}$
	as follows.
	If 
	$\xi =  (J_1, m_1)\dotplus \dots\dotplus (J_r, m_r)   \in \multi{\interval \times M}$ is a 
	representative of $\eta \in \Image [(p, \varphi) \circ \psi],$ then
	$H''_t(\eta)$ be the element of $(p^{-1}V)_s$ represented by
	$$
		(\translation{2}h'_t \translation{-2}(J_1), m_1)\dotplus \dots\dotplus (\translation{2}h'_t \translation{-2}(J_r) , m_r).
	$$
	For each $i,$ $h_t(J_i)$ is at least a degenerate interval by our construction.	
	
	It is easy to see that $H''_t$ defines a homotopy 
	$H''_t : V\times \widetilde{I}_M\to V\times \widetilde{I}_M$ between 
	$ (p, \varphi)\circ \psi $ and $id_{V\times \widetilde{I}_M}.$
	\end{proof}
\subsection{Dold-Thom criterion (2)}
	Before proceeding, we investigate a topological property of $p^{-1}z$ for $z\in BM.$
	The following notation is useful.
	Let $\elemone{ v, p, m  }$ denote the elementary configuration of type E1 given by
	$  (J, m)$ with 
	\[
		J = \intervalconvention{\bar{p}}{ -v, v }{p}
	\]
	and let
	$\elemtwo{ v, p, m }$ denote the elementary configuration of type E2 given by
	$ (K, n) \dotplus (K', n) $ with 
	\[
		K' = \left] -1 , -v \right|_{\bar{p}}, K = \,_p| v, 1 [
	\]
put
	\[ e(t,m) = \elemone{ (1-t)/2,t/ |t|, m } \]
	and
	\[ f(t,n) = \elemtwo{ t/2 , -t/|t|, n   }. \]
	Furthermore, for any $m\in M,$ let $Z(m_0)$ denote a finite sum of 
	configurations of the form $E( u_k, p_k, c_k)|_{ ( -1,1 )}$ with
	$ 1/2\leq |u| \leq 1$ and $\sum_k c_k = m_0.$ 

	For any $z \in BM,$ we have a unique representation
	$$ z = 0\otimes m_0 + t_1 \otimes m_1 + \dots + t_s\otimes m_s $$
	such that $ -1 < t_1< t_2 < \dots < t_s < 1, t_i \neq 0$ and $ m_i \neq 0 $ for $1\leq i \leq s. $

	Let $m \in M,$ and let $P(m)$ denote the set of partition of $m$ into a sum of two
	elements in $M,$ that is,
	$	P(m) = \{ (a, b) \in M_2 ~|~ a + b = m\}.	$
	Let $\alpha = ( ( a_1 , b_1 ) ,\dots, (a_s, b_s) ) \in P(m_1)\times \dots\times P(m_s),$ 
	and $F_\alpha$ denote the subspace of $\widetilde{E}_M$ which consists
	of $\xi \in \widetilde{E}_M$ such that 
	$\xi|_{(-1, 1)}$ has a representative of the form
	\begin{equation}
		Z(m_0) \dotplus \sum_{i=1}^{s} e(t_i , a_i) \dotplus	\sum_{i=1}^{s} f(t_i, b_i) \dotplus \sum_{j=1}^{q} f(u_j, n_j)
	\end{equation}
	$1 \leq |u_j| < 2$	for each $j.$
	where $0< |t_i| < 1$ for each $i$ and $1 \leq |u_j| < 2$	for each $j.$
 Then $F_\alpha $ is contained in
	$p^{-1}z.$	Similarly, let $H_\alpha$ denote the subspace of $F_\alpha$ which consists
	of $\xi$ for which $\xi|_{(-3,3)}$ has a representative of the form (1), but now  with an 
	extra condition that $q=0$
	and $f(t,b)$ are altered by $f'(t,b)$ and $Z_0(m_0)$ is of a special form as follows:
	\[
		f'(t,b) =    ( K' , b) \dotplus  (K, b )   
	\]
	with
	\[
			K' = \left] -3 , -t/2 \right|_{\bar{p}} ,  K = \,_{p}\!\left| t/2, 3\right[ ,
	\]
	where $p = t/ |t|,$
and $Z_0(m_0) =   (~\left] -3 , 3 \right[ , m_0)  .$
	Notice that $F_{\alpha}$ is not path-connected in general, but $F_\alpha$
	and $F_\beta$ can not be connected by a continuous 
	path in $p^{-1} z$ whenever $\alpha\neq \beta.$
	The same statement is true for $H_\alpha.$
	In the following lemma, $F= F_\alpha$ and $H= H_\alpha.$ 
\begin{lem}\label{lem:deformation_retraction} 
	$H$ is a deformation retract of $F.$
\end{lem}

\begin{proof}
	To construct a retraction $r : F \to H, $ we consider maps 
	$\sigma : [ 0, \infty ) \to [ 0, \infty ) $  and $\tau : [ 1/2, \infty ) \to [ 1, \infty ) $ 
	defined by
	\begin{eqnarray*}
		\sigma(v)  & = &  \left\{\begin{array}{ll}
				v & (0\leq v \leq 1/2)\\
				3v-1  & (v \geq 1/2)\\
				\end{array}\right.\\
		\tau( v ) & = & \left\{\begin{array}{ll}
			 		v+1/2  & ( 1/2\leq v \leq 3/4 )\\
					3v-1 & ( v \geq 3/4 )
			 	\end{array}\right.
	\end{eqnarray*}
	we extend $\sigma$ and $\tau$ by $\sigma( v) = -\sigma(-v)$ and $\tau(v) = - \tau( -v ) $ for negative $v.$

	For any element of $\xi \in F,$ $\xi|_{ ( -1, 1 ) } $ has a representative
	of the form
	$$		\sum_{i=1}^{s} e(t_i , a_i) \dotplus	\sum_{i=1}^{s} f(t_i, b_i)\dotplus	\sum_{i=1}^{q} f(u_j, c_j),$$
	In the following argument, we use the fact that since $b_j'$ is self-insummable,
	it is distinct from any $n_j'.$
	If $\varphi : U \to V$ is a function with gradient greater or equal to 1 defined on an open interval 
	$U\subset \RR$ with values in 
	 another interval $V\subset \RR,$ and if $h = ( J_i, m_i )  \in \multi{\interval\times M}$
	is a multiset in which each $J_i = \intervalconvention{p_i}{u_i, v_i}{q_i}$ corresponds to an interval contained in $U,$	
	then by $\varphi_* h,$ we mean a multiset $ ( J_i', m_i ),$ where 
	$J_i' = \intervalconvention{p_i}{\varphi(u_i), \varphi(v_i) }{q_i}.$
	We put
	$$
		\xi' = \sum_{i=1}^s \sigma_*(e (t_i , a_i)  ) \dotplus \sum_{i=1}^s \sigma_*(f(t_i, b_i)) %
		\dotplus \sum_{i=1}^{q} \tau_*(   f(u_j, c_j)    ).
	$$
	Then we put
	$$r(\xi ) = \xi' + \tau_*(\xi |_{ (-\infty, -1] \cup [1 , \infty ) }).$$
	This defines a deformation retraction  $ r  : F  \to H.$ 
\end{proof}

	\begin{lem}\label{lem:key-lemma2} For any $l \geq 1, $
		there exists an open set $\openset\subset F_{\filternumber} BM$ which contains 
		$F_{\filternumber-1} BM$ and
		homotopies $h_t: \openset\to \openset$ and $H_t : p^{-1}\openset\to p^{-1}\openset$ 
		such that
		\begin{enumerate}
			\item $h_0 =  id_{\openset}$, $h_t(F_{\filternumber-1} BM) \subset F_{\filternumber-1} BM$ and
			$h_1(\openset) \subset F_{\filternumber-1} BM$,
			\item $H_0 = id_{p^{-1}\openset}$ and $p\circ H_t = h_t\circ p$ for all $t$, and
			\item $H_1 : p^{-1}z \to p^{-1} h_1(z)$ is a weak homotopy equivalence for all $z\in \openset.$
		\end{enumerate}
	\end{lem}
\begin{proof}

		Let $U'$ be the subspace of $(S^1 \times M)^\filternumber$ which consists of
		elements $$((t_1, m_1),\dots, (t_\filternumber, m_\filternumber))$$ such that
		there exists at least  one $i$ with $|t_i| > 1/2.$
		Then let $U$ be the subspace of $F_j BM$  which consists of elements which is represented by
		elements in $U'.$

	We define a homotopy $h'_t : [-1, 1] \to [-1, 1]~(0 \leq t \leq 1 )$ by
	\begin{equation} \label{eqn:ht'}
		h'_t( u ) = \left\{
		\begin{array}{cc}
			-1 & -1\leq u\leq \frac{t}{2} -1\\
			2u/ (2-t) & |u| < 1-\frac{t}{2}\\
			1 & 1-\frac{t}{2} \leq u\leq 1.
		\end{array}
		\right.
	\end{equation}
	Then we define a homotopy $h_t : \openset\to \openset$ by
	$$
	( (t_1, m_1) , \dots, (t_\filternumber, m_\filternumber ))\mapsto%
	 ( ( h'_t( t_1 ), m_1  ) , \dots, (h'_t(t_\filternumber) ,  m_\filternumber  ) ).
	$$

	Next, we construct a homotopy $H_t : (p^{-1}O)_s \to (p^{-1}O)_{s+5/2}.$
	Consider a homotopy $\lambda_t : [ 0, \infty ) \to [ 0, \infty ) $ defined by
	\[
		\lambda_t(v)   =   \left\{\begin{array}{ll}
				0 & (0\leq v \leq t/4)\\
				(4v - t)/ (4-2t)  & (t/4\leq v \leq 1/2)\\
				(3t+1) v - 3t/2 & ( 1/2 \leq v \leq 1)\\
				v + 3t/2 & (v \geq 1)
			\end{array}\right.
	\]
	We also denote by $\lambda_t$ the self-homotopy of $\interval$
	 given by 	
\begin{equation}\label{eqn:lambda_t}
	\intervalconvention{p}{a,b}{q}\mapsto \intervalconvention{p}{\lambda_t(a),\lambda_t(b)}{q}.
\end{equation}
	Then let $\widetilde{\lambda}_t = \multi{ \lambda_t \times id_M }$
	be the self-homotopy of $ \multi{ \interval \times M} .$
	Similarly, using a homotopy $\nu_t : [ 0, \infty ) \to [ 0, \infty ) $ defined by
	\[
		\nu_t( v )  =  \left\{\begin{array}{ll}
			 		(t+1)v  & ( 0\leq v \leq 3/4 )\\
					(3t+1) v - 3t/2 & ( 3/4 \leq v \leq 1)\\
					v + 3t/2 & (v \geq 1)
		 	\end{array}\right._,
	\]
	we define a self-homotopy of $\interval$ by
	\begin{equation}\label{eqn:nu_t}
		\intervalconvention{p}{a,b}{q}\mapsto \intervalconvention{p}{\nu_t(a),\nu_t(b)}{q}.
	\end{equation}
	Then	we can define a self-homotopy $\widetilde{\nu}_t$ of $\multi{\interval\times M}$ by
 	$\widetilde{\nu}_t = \multi{ \nu_t \times id_M }.$
	Let $S_{-} \dotplus S_0 \dotplus S_{+} \in \multi{ \interval\times M}$ be a representative
	of a given element  $\xi \in (p^{-1}O)_s$ as in \S 4.4. Then let $H'_t(\xi)$ be
	an element of $(p^{-1}O)_{s+5/2}$ represented by
	$\widetilde{\lambda}_t( S_0 ) \dotplus \widetilde{\nu}_t( S_+ ) \dotplus \mu\widetilde{\nu}_t(S_+).$
	Now $H_t : (p^{-1}O)_s \to (p^{-1}O)_{s+5/2}$ is defined by $H_t( \xi, s ) = ( H'_t(\xi), s+5/2).$
	It is clear that $H_t$ covers $h_t.$\smallskip\\

	To complete the proof, we show that $H_1 : p^{-1}z \to p^{-1} h_1(z)$ is in fact a 
	homotopy equivalence. By renumbering $t_1, \dots , t_s,$ we may assume that 
	there exists a number $s_0$ such that
	$| t_i | > 1/2$ for $i\leq s_0$ and $|t_i| \leq 1/2$ otherwise.
	Then 	$h_1 z = 0\otimes m_0 +  h_1(t_{s_0+1}) \otimes m_{s_0+1} + \dots + h_1(t_s) \otimes m_s.$
	Let 
	\[\alpha' = ( (a_{s_0+1}, b_{s_0+1}), \dots , ( a_s , b_s ) ) \in P(m_{s_0+1})\times\dots\times P(m_s).\]
	
	By Lemma \ref{lem:deformation_retraction}, we have a deformation retraction 
	$r: F'_\alpha \to H'_\alpha.$
	We construct a map $g: H_{\alpha'} \to F_{\alpha}$ such that $g\circ r$ is
	a homotopy inverse of $H_1|_{F_\alpha}.$ Given an element of $H_{\alpha'},$
	we want to recover the information of $F_\alpha$ by using the data given by $\alpha.$

	For this, we insert ``the standard lift of $z$ of type $\alpha$'' 
	defined by the multiset
	\begin{equation}
		\zeta = Z'(m_0) \dotplus \sum_{i=1}^{s} e(t_i , a_i) \dotplus	\sum_{i=1}^{s} f'(t_i, b_i)
	\end{equation}
	at the origin, where 
	$Z'(m_0) =  ( [ -1, 1 [  , m_0 )$
	and $f'(t,b)$ denotes a configuration which is constructed by
	changing the parity of only one endpoint among four occurring in $f(t,b)$ so that
	the resulting configuration consists of two half-open intervals.
	For $\xi \in H_{\alpha'},$ let $S'_0 \dotplus S_+$ be a representative of the 
	positive part of $\xi.$
	Then we may assume that 
	$$S_0'  =  (K_0, m_0)\dotplus (K_1,m_1) \dotplus \dots \dotplus(K_r, m_r)$$
	 with $u_L(K_i) = 0, p_L(K_i) = 1.$
	For $1\leq i\leq r,$ we alter $p_L(K_i)$ by the value $-p_R(K_i)$ and
	denote by $S_0''$ the resulting multiset.
	Now if we define
	$g : H_{\alpha'}\to F_\alpha$  by
	$$
		g(\xi ) = \zeta \dotplus \mu ( \translation2 ( S_0'' \dotplus S_+ )).
	$$
	then it is easy to see that
	$  g\circ r$ is a homotopy inverse to $H_1 |_{F_\alpha}.$
\end{proof}

\begin{proof}[Proof of Propostion \ref{prop:quasifibration}]
	$F_0 BM = \{*\}$ is obviously a distinguished set with respect to $p.$
	By induction, assume that $F_{l-1} BM$ is a distinguished set. Then 
	by Lemma \ref{lem:key-lemma2} and Hilfssatz 2.10 of 
	\cite{dold-thom-quasifaserungen-und-unendliche}, we have an open set
	$O\subset F_{l}$ which contains $F_{l-1} BM,$ which is a distinguished set.
	Then by Lemma \ref{lem:key-lemma1} and 
	Satz 2.2 of \cite{dold-thom-quasifaserungen-und-unendliche},
	we have that $F_{l} BM$ is a distinguished set. Then by Satz 2.15 of 
	\cite{dold-thom-quasifaserungen-und-unendliche}, 
	$p : \widetilde{E}_M \to BM$ is a quasifibration.
\end{proof}

\bibliographystyle{plain}
\bibliography{reference}

\end{document}